\documentclass[12pt]{article}
\usepackage{amssymb,amsmath,amsthm,amscd}

\title{Size of Union}
\author{Yuanzhong Ou, Boli Wang, Min Yan \\ Hong Kong University of Science and Technology}

\newcommand{\sub}{\subset}

\newtheorem{theorem}{Theorem}

\newtheorem{proposition}[theorem]{Proposition}

\newtheorem*{theorem*}{Theorem}
\newtheorem*{proposition*}{Proposition}
\newtheorem*{addendum*}{Addendum}

\theoremstyle{definition}
\newtheorem*{remark*}{Remark}

\begin{document}

\maketitle

\section{Main Result}

Given finite sets $A_1,A_2,\dotsc,A_n$ with respective numbers $a_1,a_2,\dotsc,a_n$ of elements, the union $A_1\cup A_2\cup\dotsb\cup A_n$ can have as many as $a_1+a_2+\dotsb+a_n$ elements and as few as $\max\{a_1,a_2,\dotsc,a_n\}$ elements. The maximum is realised when the sets are pairwise disjoint. When the minimum is realised, chances are there are many nonempty intersections among the sets. 

In this paper, we fix $k\le n$ and study the bound on the size of the union under the additional assumption that the intersection of any $k$ sets is empty. For $k=2$, this is the trivial pairwise disjoint case. 

In a simpler version of the problem, the sets are Lebesgue measurable subsets of some Euclidean space, and the size is the Lebesgue measure. The problem is simpler because any non-negative number is allowed to be the size, not just non-negative {\em integers}. 

\begin{theorem*}
Let non-negative numbers $a_1,a_2,\dotsc,a_n$ be given. Let $2\le k\le n$ and 
\[
\bar{a}=\frac{1}{k-1}(a_1+a_2+\dotsb+a_n).
\]
Then there are Lebesgue measurable subsets $A_1,A_2,\dotsc,A_n$, such that $\mu(A_i)=a_i$, $\mu(\cup A_i)=a$, and the intersection of any $k$ subsets among $A_i$ is empty, if and only if 
\[
\max\{a_1,a_2,\dotsc,a_n,\bar{a}\}\le a\le a_1+a_2+\dotsb+a_n.
\]
Moreover, if $a_1,a_2,\dotsc,a_n$ and $a$ are integers, then the same holds for the case $A_i$ are finite sets and $\mu$ counts the number of elements.
\end{theorem*}

The bounds for $a$ in the theorem are well known for the measure case. By taking convex combinations of sizes of pure intersections (see Section \ref{lower}), it is not hard to see that, if $a$ and $a'$ are realised as the sizes of unions, then any number between $a$ and $a'$ can also be realised as the size of a union. So the new claim here is the realisability of the two bounds (especially the lower bound) and any number between the two bounds. Moreover, in an addendum in Section \ref{lower}, we will further specify how the realisation can be constructed in the ``most efficient'' way. 

We believe the theorem was not known for the case of counting the number of elements. The case is more subtle because we need to make sure that all the sizes in the realisation are non-negative integers. 

The measure part of the theorem remains true for any measure space $(X,\mu)$ with the property that $\mu(X)=\infty$, and for any $A\sub X$ of finite measure and any $b>0$, there is a measurable $B\sub X$, such that $A\cap B=\emptyset$ and $\mu(B)=b$. A suitable probabilistic version of the theorem is also not hard to state and prove.

Our theorem is a very simple case of Boolean probability bounding problem \cite[Chapter 19]{boole} that asks the question that, if one knows the probability of some logical combinations of events, how much one can say about the probability of another logical combination. In the theorem, we know the probability of the single events and that $k$ events cannot happen at the same time (i.e., the probability of such combinations are zero), and the answer is the exact range about the probability that at least one event happens. Lots of research have been done on the problem. See \cite{hoppe,pg,ve} for some of the latest developments. However, these works are usually based on the linear programming method \cite{hal}, and the bounds are often optimal for some choices of $a_i$ but never for all choices. As far as we know, the only construction that realises all the individual $a_i$ as the measure of $A_i$ is by Fr\'echet \cite{frechet}. Fr\'echet's work is our theorem without the assumption on the emptyness of the intersection.

\section{The Lower Bound}
\label{lower}

The bounds in the theorem are the well known Bonferroni type inequalities \cite{bon}. The only less trivial one is $\mu(\cup A_i)\ge \bar{a}$. We will give the proof here, mainly for the purpose of explaining the addendum to the main result.

For distinct $1\le i_1,i_2,\dotsc,i_l\le n$, we introduce ``pure intersections'' 
\begin{align*}
B_{i_1i_2\dotsb i_l}
&=A_{i_1}\cap A_{i_2}\cap\dotsb\cap A_{i_l}-\cup_{j\ne i_1,i_2,\dotsc,i_l}A_j \\
&=A_{i_1}\cap A_{i_2}\cap\dotsb\cap A_{i_l}-\cup_{j\ne i_1,i_2,\dotsc,i_l}A_{i_1}\cap A_{i_2}\cap\dotsb\cap A_{i_l}\cap A_j.
\end{align*}
The theorem assumes $B_{i_1i_2\dotsb i_l}=\emptyset$ for $l\ge k$.  Therefore we have disjoint union decompositions
\begin{align*}
A_j
&=B_j\sqcup(\sqcup_{i\ne j} B_{ij})\sqcup(\sqcup_{\substack{i_1,i_2\ne j \\ i_1<i_2}} B_{i_1i_2j})\sqcup\dotsb\sqcup(\sqcup_{\substack{i_1,\dotsc,i_{k-2}\ne j \\ i_1<\dotsb<i_{k-2}}} B_{i_1\dotsb i_{k-2}j}), \\
A_1\cup\dotsb\cup A_n
&=(\sqcup_iB_i)\sqcup(\sqcup_{i_1<i_2} B_{i_1i_2})\sqcup(\sqcup_{i_1<i_2<i_3} B_{i_1i_2i_3})\sqcup\dotsb\sqcup(\sqcup_{i_1<\dotsb<i_{k-1}} B_{i_1\dotsb i_{k-1}}).
\end{align*}
This implies
\begin{align*}
\mu(A_j)
&=\mu(B_j)+\sum_{i\ne j}\mu(B_{ij})+\sum_{\substack{i_1,i_2\ne j \\ i_1<i_2}}\mu(B_{i_1i_2j})+\dotsb+\sum_{\substack{i_1,\dotsc,i_{k-2}\ne j \\ i_1<\dotsb<i_{k-2}}}\mu(B_{i_1\dotsb i_{k-2}j}), \\
\mu(A_1\cup\dotsb\cup A_n)
&=\sum_i\mu(B_i)+\sum_{i_1<i_2}\mu(B_{i_1i_2})+\sum_{i_1<i_2<i_3}\mu(B_{i_1i_2i_3})+\dotsb+\sum_{i_1<\dotsb<i_{k-1}}\mu(B_{i_1\dotsb i_{k-1}}).
\end{align*}
Adding the first equality together for various $j$ and comparing with the second equality, we get
\begin{align*}
&\mu(A_1)+\mu(A_2)+\dotsb+\mu(A_n) \\
&=\sum_i\mu(B_i)+2\sum_{i_1<i_2}\mu(B_{i_1i_2})+3\sum_{i_1<i_2<i_3}\mu(B_{i_1i_2i_3})+\dotsb+(k-1)\sum_{i_1<\dotsb<i_{k-1}}\mu(B_{i_1\dotsb i_{k-1}}) \\
&\le (k-1)\left(\sum_i\mu(B_i)+\sum_{i_1<i_2}\mu(B_{i_1i_2})\sum_{i_1<i_2<i_3}\mu(B_{i_1i_2i_3})+\dotsb+\sum_{i_1<\dotsb<i_{k-1}}\mu(B_{i_1\dotsb i_{k-1}})\right) \\
&= (k-1)\mu(A_1\cup\dotsb\cup A_n).
\end{align*}

The proof tells us that the lower bound $\bar{a}$ is realised if and only if
\[
\mu(B_i)=\mu(B_{i_1i_2})=\mu(B_{i_1i_2i_3})=\dotsb=\mu(B_{i_1\dotsb i_{k-2}})=0.
\]
This means that the pure intersections of $j$ subsets are almost empty for any $j\ne k-1$. In other words, the elements of $A_i$ are ``concentrated'' in the pure intersections of $k-1$ subsets. 

Let $\sigma=a_1+a_2+\dotsb+a_n$. Consider the sequence
\[
\sigma>\frac{\sigma}{2}>\dotsb>\frac{\sigma}{n-1}>\frac{\sigma}{n}.
\]
We have
\[
\sigma>\frac{\sigma}{2}>\dotsb>\frac{\sigma}{m-1}\ge \max\{a_1,a_2,\dotsc,a_n\}>\frac{\sigma}{m}
\]
for some $m\le n$. For any $k\le m$, we expect the critical case $\mu(\cup A_i)=\frac{\sigma}{k-1}$ to be realisable by pure intersections of $k-1$ subsets. Now if the size of the union lies between two critical cases, then we expect the realisation can also be constructed ``in between''.

\begin{addendum*}
If
\begin{equation}\label{add}
\frac{1}{k-2}(a_1+a_2+\dotsb+a_n)\ge a\ge \frac{1}{k-1}(a_1+a_2+\dotsb+a_n)\ge \max\{a_1,a_2,\dotsc,a_n\},
\end{equation}
then it is possible to find $A_i$, such that $\mu(A_i)=a_i$, $\mu(\cup A_i)=a$, and the pure intersections of $j$ subsets are empty for $j\ne k-1,k-2$.
\end{addendum*}

The addendum holds only for the measure. At the end of the paper, we will construct an example that shows that the addendum does not hold for counting.

\section{Realisation for Measure} 
\label{smeasure}

In this section, we prove that the lower bound in the main theorem can be realised. Without loss of generality, we will always assume 
\begin{equation}\label{assume}
a_1\le a_2\le\dotsb\le a_n.
\end{equation}

We first consider the case $\bar{a}\le \max\{a_1,a_2,\dotsc,a_n\}=a_n$. This means that
\[
a_n\ge \bar{a}'=\frac{1}{k-2}(a_1+a_2+\dotsb+a_{n-1}).
\]
Note that $\max\{a_1,a_2,\dotsc,a_{n-1},\bar{a}'\}=\max\{a_{n-1},\bar{a}'\}$ is the lower bound for the case $k-1\le n-1$. We may try to apply the induction here. The initial case of the induction is $k=2<n$. In the initial case, we have $\bar{a}=a_1+a_2+\dotsb+a_n$, and $\mu(\cup_{i=1}^{n-1} A_i)=\bar{a}=\max\{a_1,a_2,\dotsc,a_n,\bar{a}\}$ always holds. So by induction, we can find $A_1,A_2,\dotsc,A_{n-1}$, such that 
\[
\mu(A_i)=a_i,\quad
\mu(\cup_{i=1}^{n-1} A_i)=\max\{a_{n-1},\bar{a}'\},
\]
and the intersection of any $k-1$ subsets is empty. Let $\langle x\rangle$ be a subset of measure $x$ and introduce (note that $a_n\ge \max\{a_{n-1},\bar{a}'\}$)
\[
A_n=(A_1\cup A_2\cup \dotsb\cup A_{n-1})\sqcup \langle a_n-\max\{a_{n-1},\bar{a}'\}\rangle.
\]
Then among $A_1,A_2,\dotsc,A_{n-1},A_n$, we have 
\[
\mu(A_n)=\mu(\cup_{i=1}^{n} A_i)=\mu(\cup_{i=1}^{n-1} A_i)+(a_n-\max\{a_{n-1},\bar{a}'\})=a_n,
\]
and the intersection of any $k$ subsets is empty.

Next we turn to the case $\bar{a}\ge a_n$. This means that $b=\bar{a}-a_n\ge 0$, and we have
\[
a_n=\frac{1}{k-1}((a_1-b)+\dotsb+(a_{k-1}-b)+a_k+\dotsb+a_n).
\]
If $b\le a_1$, then for the problem of realising the lower bound for $n$ subsets of measure $a'_1=a_1-b$, $\dotsc$, $a'_{k-1}=a_{k-1}-b$, $a'_k=a_k$, $\dotsc$, $a'_n=a_n$, such that the intersection of any $k$ subsets is empty, we have
\[
a_n=\max\{a'_1,a'_2,\dotsc,a'_n\}
=\frac{1}{k-1}(a'_1+a'_2+\dotsb+a'_n).
\]
This fits into the case $\bar{a}\le a_n$ we proved earlier. Therefore we can find $A'_1,A'_2,\dots,A'_n$, such that 
\[
\mu(A'_i)=a'_i,\quad
\mu(\cup A'_i)=a_n=\bar{a}-b,
\]
and the intersection of any $k$ subsets is empty. Take
\[
A_i=\begin{cases}
A'_i\sqcup\langle b \rangle, &\text{if }1\le i<k, \\
A'_i, &\text{if }k\le i\le n.
\end{cases}
\]
Then $\mu(A_i)=a_i$, $\mu(\cup A_i)=\mu(\cup A'_i)+b=\bar{a}$, and the intersection of any $k$ subsets among $A_i$ is still empty.

If $b\ge a_1$, then subtracting $b$ from $a_i$ may yield negative number. So we subtract $a_1$ instead to get $a'_2=a_2-a_1,\dotsc,a'_{k-1}=a_{k-1}-a_1,a'_k=a_k,\dotsc,a'_n=a_n$. Consider the problem of realising the lower bound for $n-1$ subsets of measure $a'_2,a'_3,\dotsc,a'_n$, such that the intersection of any $k$ subsets is empty. We have
\[
a_n=\max\{a'_2,a'_3,\dotsc,a'_n\}
\le\frac{1}{k-1}(a'_2+a'_3+\dotsb+a'_n)=\bar{a}-a_1.
\]
Now we are in the situation of realising $n-1$ subsets such that the intersection of any $k$ subsets is empty. Again we may try to apply the induction. Since we keep the same $k$ and reduce $n$, the initial case is $k=n$. Moreover, we have the additional property that $\max\{a_1,a_2,\dotsc,a_n\}\le \bar{a}$. So the initial case is covered by the following result.

\begin{proposition}\label{extreme}
Suppose $a_i\ge 0$ satisfy
\begin{equation}\label{condn}
\max\{a_1,a_2,\dotsc,a_n\}\le \bar{a}=\frac{1}{n-1}(a_1+a_2+\dotsb+a_n).
\end{equation}
Then there are Lebesgue measurable subsets $A_i$, such that 
\[
\mu(A_i)=a_i,\quad
\mu(\cup_{i=1}^n A_i)=\bar{a},\quad
\cap_{i=1}^nA_i=\emptyset.
\]
\end{proposition}

\begin{proof}
We expect the lower bound to be realised when the only nonempty pure intersections are those of $n-1$ subsets
\[
C_i=B_{1\dotsb (i-1)(i+1)\dotsb n}
=A_1\cap\dotsb\cap A_{i-1}\cap A_{i+1}\cap \dotsb\cap A_n.
\]
The construction is then to find pairwise disjoint $C_i$ and take
\[
A_i=C_1\sqcup\dotsb\sqcup C_{i-1}\sqcup C_{i+1}\sqcup\dotsb\sqcup C_n.
\]

Let $x_i=\mu(C_i)$. Then we can find suitable $C_i$ if and only if the system of linear equations
\[
x_1+\dotsb+x_{i-1}+x_{i+1}+\dotsb+x_n=a_i,\quad i=1,2,\dotsc,n,
\]
has non-negative solution. The system has unique solution $x_i=\bar{a}-a_i$. The condition for the solutions to be non-negative is exactly \eqref{condn}.
\end{proof}

Continuing the proof, by induction, we find $A'_2,A'_3,\dots,A'_n$, such that 
\[
\mu(A'_i)=a'_i,\quad
\mu(\cup A'_i)=\bar{a}-a_1,
\]
and the intersection of any $k$ subsets is empty. Take
\[
A_i=\begin{cases}
\langle a_1 \rangle, &\text{if }i=1, \\
A'_i\sqcup\langle a_1 \rangle, &\text{if }2\le i<k, \\
A'_i, &\text{if }k\le i\le n.
\end{cases}
\]
Then $\mu(A_i)=a_i$, $\mu(\cup A_i)=\mu(\cup A'_i)+a_1=\bar{a}$, and the intersection of any $k$ subsets from $A_i$ is still empty.

Finally, we prove the addendum in Section \ref{lower}. Suppose \eqref{add} is satisfied. We have subsets $A'_1,A'_2,\dots,A'_n$, such that
\[
\mu(A'_i)=a_i,\quad
\mu(\cup A'_i)=\frac{\sigma}{k-1},
\]
and only the pure intersections of $k-1$ subsets are nonempty. We want to increase the size of the union to $a$ while keeping the size of each subset to be still $a_i$. Moreover, we want to accomplish this by ``leaking'' some size from the pure intersections of $k-1$ subsets to pure intersections of $k-2$ subsets. 

Specifically, for any $0\le x\le \mu(B'_{i_1i_2\dotsb i_{k-1}})$, we take 
\begin{align*}
B_{i_1i_2\dotsb i_{k-1}}
&=B'_{i_1i_2\dotsb i_{k-1}}-\langle x\rangle,  \\
B_{i_1\dotsb i_{p-1}i_{p+1}\dotsb i_{k-1}}
&=\langle \frac{x}{k-2}\rangle,
\quad 1\le p\le k-1,
\end{align*}
and keep all other pure intersections the same. For $j\ne i_q$, the pure intersections that form $A_j$ are not changed, so that $A_j=A'_j$ and 
\[
\mu(A_j)=\mu(A'_j)=a_j.
\]
On the other hand, we have $A_{i_q}=(A'_{i_q}-\langle x\rangle)\sqcup(\sqcup_{p\ne q} B_{i_1\dotsb i_{p-1}i_{p+1}\dotsb i_{k-1}})$, so that
\[
\mu(A_{i_q})=\mu(A'_{i_q})-x+(k-2)\frac{x}{k-2}=a_{i_q}.
\]
Moreover, we have $\cup A_i=(\cup A'_i-\langle x\rangle)\sqcup(\sqcup_p B_{i_1\dotsb i_{p-1}i_{p+1}\dotsb i_{k-1}})$, so that 
\[
\mu(\cup A_i)=\mu(\cup A'_i)-x+(k-1)\frac{x}{k-2}=\frac{\sigma}{k-1}+\frac{x}{k-2}.
\]

The leaking of size $x$ described above can be carried out independently for all pure intersections of $k-1$ subsets. Suppose we choose $0\le x_{i_1i_2\dotsb i_{k-1}}\le \mu(B'_{i_1i_2\dotsb i_{k-1}})$ for all pure intersections of $k-1$ subsets and construct
\begin{align*}
B_{i_1i_2\dotsb i_{k-1}}
&=B'_{i_1i_2\dotsb i_{k-1}}-\langle x_{i_1i_2\dotsb i_{k-1}}\rangle, \\
B_{i_1i_2\dotsb i_{k-2}}
&=\langle \frac{1}{k-2}\sum_{j\ne i_1,i_2,\dotsc,i_{k-2}}x_{i_1i_2\dotsb i_{k-2}j}\rangle,
\end{align*}
and keep all the other pure intersections empty. Then we still have $\mu(A_i)=a_i$ and
\[
\mu(\cup A_i)
=\mu(\cup A'_i)+\frac{1}{k-2}\sum x_{i_1i_2\dotsb i_{k-1}}
=\frac{\sigma}{k-1}+\frac{1}{k-2}\sum x_{i_1i_2\dotsb i_{k-1}}.
\]
The sum $\sum x_{i_1i_2\dotsb i_{k-1}}$ can be any non-negative number $\le\sum \mu(B'_{i_1i_2\dotsb i_{k-1}})=\mu(\cup A'_i)=\frac{\sigma}{k-1}$. Therefore by choosing suitable $x_{i_1i_2\dotsb i_{k-1}}$, $\mu(\cup A_i)$ can be any number between $\frac{\sigma}{k-1}$ and 
\[
\frac{\sigma}{k-1}+\frac{1}{k-2}\frac{\sigma}{k-1}=\frac{\sigma}{k-2}.
\]

\section{Realisation for Counting}

In this section, we try to modify the proof of the measure version of the main theorem to the counting version. The proof for the case $\bar{a}\le a_n$ is valid for the counting version if we take $\bar{a}'$ to be the smallest integer $\ge \frac{1}{k-2}(a_1+a_2+\dotsb+a_{n-1})$. For the case $\bar{a}\ge a_n$, we need to realise the smallest integer $\ge \bar{a}$ by subsets of integer sizes. Of course, the ideal case would be that $\bar{a}$ is already an integer, which means that $a_1+a_2+\dotsb+a_n$ is divisible by $k-1$. It tuns out that the general case can be reduced to the ideal case.

Here is the reason for reducing the general case. Without loss of generality, we may assume \eqref{assume} holds. If $a_1=0$, then the realisation is actually for the same $k$ but with smaller $n$. If we keep getting $a_i=0$, the induction will reduce to the initial case $k=n$. If we still have $a_1=0$ in the initial case $k=n$, then by \eqref{assume},
\[
a_n\ge \dfrac{1}{n-1}(a_2+\dotsb+a_n)=\bar{a}.
\]
By the assumption $\bar{a}\ge a_n$, we find $\bar{a}=a_n$ is an integer. 

So we may further assume $a_1>0$ in addition to \eqref{assume}. Suppose $0<r<k-1$ is the remainder of the division of $a_1+a_2+\dotsb+a_n$ by $k-1$. Then the integer part of $\bar{a}$ is
\[
\bar{a}'=\dfrac{1}{k-1}((a_1-1)+\dotsb+(a_r-1)+a_{r+1}+\dotsb+a_n),
\]
and $\bar{a}\ge a_n$ implies $\bar{a}'\ge a_n$. If the ideal cases can be realised, then we have finite sets $A'_1,\dotsc,A'_n$, such that  
\[
\mu(A'_i)=\begin{cases}
a_i-1, &\text{if }1\le i\le r, \\
a_i, &\text{if }r< i\le n,
\end{cases} \qquad
\mu(\cup A'_i)=\bar{a}',
\]
and the intersection of any $k$ sets is empty. Take
\[
A_i=\begin{cases}
A'_i\sqcup\langle 1 \rangle, &\text{if }1\le i\le r, \\
A'_i, &\text{if }r< i\le n.
\end{cases}
\]
Then $\mu(A_i)=a_i$, and $\mu(\cup A_i)=\mu(\cup A'_i)+1=\bar{a}'+1$ is the smallest integer $\ge \bar{a}$. Moreover, the intersection of any $k$ sets from $A_i$ is empty.

Once we reduce the proof of the case $\bar{a}\le a_n$ to the ideal case that $\bar{a}$ is already an integer, the rest of the proof for the measure version remains valid, because all the numbers appearing in the proof are integers. This concludes the proof for the realisation of the lower bound of the number of elements in the union of finite sets.

To show that any number between the lower and upper bounds can be realised, we only need to show that if $a$ and $a+1$ are between the bounds, and $a$ is realised, then $a+1$ is also realised. So assume we have finite sets $A'_i$ satisfying $\mu(A'_i)=a_i$, $\mu(\cup A'_i)=a$, and the intersection of any $k$ sets is empty. Since $\mu(\cup A'_i)<a+1\le a_1+a_2+\dotsb+a_n$, some pure intersection $B'_{i_1i_2\dotsb i_l}\ne\emptyset$ with $l\ge 2$. Fix any $1\le p<l$ and construct
\[
B_{i_1i_2\dotsb i_l}=B'_{i_1i_2\dotsb i_l}-\langle 1\rangle,\quad
B_{i_1i_2\dotsb i_p}=B'_{i_1i_2\dotsb i_p}\sqcup \langle 1\rangle,\quad
B_{i_{p+1}i_{p+2}\dotsb i_l}=B'_{i_{p+1}i_{p+2}\dotsb i_l}\sqcup \langle 1\rangle,
\]
where the three single element sets $\langle 1\rangle$ are distinct. We also keep all the other pure intersections to be the same. Then $\mu(A_i)=\mu(A'_i)$ and $\mu(\cup A_i)=\mu(\cup A'_i)+1=a+1$. Moreover, since we only modify pure intersections of less than $k$ sets, the pure intersections of $k$ sets from $A_i$ are still empty.

Finally, we construct an example showing the addendum does not hold for counting. Consider $a_1=a_2=\dotsb=a_n=1$ and $k=n$. We have $\sigma=n$ and $\dfrac{\sigma}{n-2}>2>\dfrac{\sigma}{n-1}$ whenever $n>4$. If each $A_i$ contains one element and $\cup A_i$ contains two elements, then without loss of generality, we may assume
\[
A_1=\dotsb=A_r=\{x\},\quad
A_{r+1}=\dotsb=A_n=\{y\},\quad x\ne y,\quad 1\le r\le n.
\]
This shows that the only nonempty pure intersections are $B_{1\dotsb r}=\{x\}$ and  $B_{(r+1)\dotsb n}=\{y\}$.

\end{document}